\documentclass[12pt]{amsart}
\usepackage[osf,sc]{mathpazo}
\usepackage{amssymb}

\usepackage{geometry}
\usepackage{bm}
\usepackage{bbm}
\usepackage{graphicx}
\usepackage{hyperref}
\hypersetup{
    colorlinks=true, 
    linktoc=all,     
    linkcolor=blue,
    citecolor=red,
    filecolor=black,
    urlcolor=blue	
}
\usepackage{enumerate}
\usepackage[inline]{enumitem}
\makeatletter
\newcommand{\inlineitem}[1][]{%
\ifnum\enit@type=\tw@
    {\descriptionlabel{#1}}
  \hspace{\labelsep}%
\else
  \ifnum\enit@type=\z@
       \refstepcounter{\@listctr}\fi
    \quad\@itemlabel\hspace{\labelsep}%
\fi} \makeatother
\newcommand{\gl}{\lambda}

\newcommand{\Gl}{\Lambda}

\newcommand{\Gs}{\Sigma}

\newcommand{\Gom}{\Omega}

\newcommand{\subs}{\subset}

\newcommand{\sbnq}{\subsetneq}

\newcommand{\mbb}{\mathbb}

\newcommand{\mcl}{\mathcal}

\newcommand{\ol}{\overline}

\newcommand{\us}{\underset}
\newcommand{\os}{\overset}

\newcommand{\lra}{\longrightarrow}

\newcommand{\I}{\mcl I}

\newcommand{\N}{\mbb N}
\newcommand{\Z}{\mbb Z}
\newcommand{\R}{\mcl R}

\newcommand{\eqdef}{\overset{\mathrm{def}}{=\joinrel=}}

\newcommand{\equ}[1]{%
\begin{equation*}
#1
\end{equation*}
}
\newcommand{\equa}[1]{%
\begin{equation*}
\begin{aligned}
#1
\end{aligned}
\end{equation*}
}

\DeclareMathOperator{\Diag}{Diag}
\newcommand{\mattwo}[4]{%
\begin{pmatrix}
  #1 & #2\\ #3 & #4
\end{pmatrix}
}

\newcommand{\matthree}[9]{%
\begin{pmatrix}
  #1 & #2 & #3\\ #4 & #5 & #6\\ #7 & #8 & #9
\end{pmatrix}
}

\theoremstyle{plain}
\newtheorem{theorem}{Theorem}[section]

\newtheorem{lemma}[theorem]{Lemma}

\makeatletter
\def\namedlabel#1#2{\begingroup
	\def\@currentlabel{#2}%
	\label{#1}\endgroup
}
\makeatother

\newtheorem*{thmOmega}{\bf{Theorem} $\bm{\Gom}$}
\newtheorem*{thmSigma}{\bf{Theorem} $\bm{\Gs}$}

\theoremstyle{definition}
\newtheorem{defn}[theorem]{Definition}

\theoremstyle{remark}
\newtheorem{remark}[theorem]{Remark}
\newtheorem{example}[theorem]{Example}
\numberwithin{equation}{section}
\begin{document}
\title[On the Surjectivity of Certain Maps IV: For Congruence Ideal Subgroups]{On the Surjectivity of Certain Maps IV: For Congruence Ideal Subgroups of Type $A_k$ and $C_k$}
\author[C.P. Anil Kumar]{C.P. Anil Kumar}
\address{Post Doctoral Fellow in Mathematics, Room. No. 223, Middle Floor, Main Building, Harish-Chandra Research Institute, Chhatnag Road, Jhunsi, Prayagraj (Allahabad)-211019, Uttar Pradesh, INDIA}
\email{akcp1728@gmail.com}
\thanks{This work is done while the author is a Post Doctoral Fellow at Harish-Chandra Research Institute, Prayagraj(Allahabad), INDIA.}
\begin{abstract}
For a positive integer $k$, we extend the surjectivity results from special linear groups (Type $A_k$) and symplectic linear groups (Type $C_k$) onto product of generalized projective spaces by associating the rows or columns, to certain congruence ideal subgroups of special linear groups and symplectic linear groups. 
\end{abstract}
\maketitle

\section{\bf{Introduction}}
In this article we extend the main results, namely the surjectivity Theorem $\Gl$ and the surjectivity Theorem $\Gs$ of C.~P.~Anil Kumar~\cite{CPAKII}, to certain congruence ideal subgroups (Definition~\ref{defn:CIS}) of special linear groups of degree $k+1$ (Type $A_k$) and symplectic linear groups of degree $2k$ (Type $C_k$) for $k\in \N$. We also give a sufficient number of examples to explore the aspects of the strong approximation property SAP and the unital set condition USC for an ideal, through which we can understand these notions better. Some of the examples such as~\ref{Example:NotUSC},~\ref{Example:NotSAP} can be understood with the help of some interesting facts about the topology of certain spaces. In Section~\ref{sec:SAP}, we relate the SAP to GE-rings as defined by P.~M.~Cohn~\cite{MR0207856} in Theorems~\ref{theorem:SAP},~\ref{theorem:SAPConverse}. 
\section{\bf{The Main Results}}
We begin with a few definitions. 


\begin{defn}
	\label{defn:ProjSpaceRelation}
	~\\
	Let $\R$ be a commutative ring with unity. Let $k\in \mbb{N}$ and \equ{\mcl{GCD}_{k+1}(\R)=\{(a_0,a_1,a_2,\ldots,a_k)\in \R^{k+1}\mid \us{i=0}{\os{k}{\sum}}\langle a_i\rangle=\R\}.} Let $\I \subsetneq \R$ be an ideal
	and $m_0,m_1,\ldots,m_k\in \mbb{N}$. Define an equivalence relation \equ{\sim^{k,(m_0,m_1,\ldots,m_k)}_{\I}} on $\mcl{GCD}_{k+1}(\R)$
	as follows. For \equ{(a_0,a_1,a_2,\ldots,a_k),(b_0,b_1,b_2,\ldots,b_k) \in \mcl{GCD}_{k+1}(\R)} we say \equ{(a_0,a_1,a_2,\ldots,a_k)\sim^{\{k,(m_0,m_1,\ldots,m_k)\}}_{\I}(b_0,b_1,b_2,\ldots,b_k)} if there exists a \equ{\gl\in \R\text{ with }\ol{\gl}\in \bigg(\frac{\R}{\I}\bigg)^{*}} such that we have 
	\equ{a_i \equiv \gl^{m_i}b_i \mod \I,0\leq i\leq k.}  
\end{defn}

\begin{defn}
	\label{defn:GenProjSpace}
	Let $\R$ be a commutative ring with unity. Let $k\in \mbb{N}$ and \equ{\mcl{GCD}_{k+1}(\R)=\{(a_0,a_1,a_2,\ldots,a_k)\in \R^{k+1}\mid \us{i=0}{\os{k}{\sum}}\langle a_i\rangle=\R\}.}
	Let $m_0,m_1,\ldots,m_k\in \mbb{N}$ and $\I \subsetneq \R$ be an ideal. Let $\sim^{k,(m_0,m_1,\ldots,m_k)}_{\I}$ denote the equivalence relation as in Definition~\ref{defn:ProjSpaceRelation}. Then we define 
	\equ{\mbb{PF}^{k,(m_0,m_1,\ldots,m_k)}_{\I}\eqdef \frac{\mcl{GCD}_{k+1}(\R)}{\sim^{k,(m_0,m_1,\ldots,m_k)}_{\I}}.}
	If $\mcl{I}=\R$ then let $\sim^{k,(m_0,m_1,\ldots,m_k)}_{\I}$ be the trivial equivalence relation on $\mcl{GCD}_{k+1}(\R)$ where any two elements are related.
	We define \equ{\mbb{PF}^{k,(m_0,m_1,\ldots,m_k)}_{\I} \eqdef \frac{\mcl{GCD}_{k+1}(\R)}{\sim^{k,(m_0,m_1,\ldots,m_k)}_{\I}}}
	a singleton set having single equivalence class.
\end{defn}
\begin{defn}
	\label{defn:unitalset}
	Let $\R$ be a commutative ring with unity. Let $k\in \mbb{N}$. We say a finite subset 
	\equ{\{a_1,a_2,\ldots,a_k\}\subs \R} consisting of $k\operatorname{-}$elements (possibly with repetition) 
	is unital or a unital set if the ideal generated by the elements of the set is a unit ideal.
\end{defn}


\begin{defn}
	\label{defn:UnitalSetCond} Let $\R$ be a commutative ring with unity. Let $k\in \mbb{N}$ and $\mcl{I} \sbnq \mcl{\R}$ be an ideal. 
	We say $\mcl{I}$ satisfies the unital set condition USC if for every unital set
	$\{a_1,a_2,\ldots,a_k\} \subs \R$ with $k \geq 2$, there exists an
	element $b \in \langle a_2,\ldots,a_k\rangle$ such that $a_1+b$ is a unit modulo
	$\mcl{I}$.
\end{defn}
\begin{example}
	In the ring $\Z$ any ideal $\langle 0\rangle \neq \mcl{I}\subsetneq \Z$ satisfies the USC using Lemma $4.1$ in C.~P.~Anil Kumar~\cite{CPAKII}. The zero ideal in $\Z$ does not satisfy the USC. More generally, in a commutative ring $\R$ with unity, any ideal $\mcl{I}\sbnq \R$ which is contained in only a finitely many maximal ideals of $\R$, satisfies the USC using Proposition $4.4$ in~\cite{CPAKII}. So any non-zero ideal $\I$ in a Dedekind domain $\R$ satisfies the USC. Also see Example~\ref{Example:USC} for some exotic variety of examples where the ideals $\langle 0\rangle \neq \I\sbnq \R$ in these examples are not contained in only a finitely many maximal ideals but satisfy the USC. See another Example~\ref{Example:NotUSC2} for an ideal $\I\sbnq \R$ which does not satisfy the USC. For difficult examples of ideals $\I\sbnq \R$ which do not satisfy the USC, see Examples~\ref{Example:NotUSC},~\ref{Example:NotSAP}.
\end{example}
Now we give examples of rings $\mcl{S}$ and ideals $\mcl{J}\sbnq \R$ which satisfy the USC but $\mcl{J}$ is contained in infinitely many maximal ideals of $\R$. 
\begin{example}
	\label{Example:USC}
	~\\
		\begin{enumerate}[label=(\roman*)]
			\item 
			For this let $\mathbb{K}$ be a field and $\mcl{S}=\us{i\in \N}{\prod} \mathbb{K}_i$, an infinite direct product of fields where $\mathbb{K}_i=\mathbb{K},i\in\N$. Let $\mcl{J}$ be the zero ideal. Then $\mcl{J}$ is contained in infinitely many maximal ideals and $\mcl{J}$ satisifes the USC.
			
			\item Here let $\mcl{O}_i,i\in \N$ be an infinite collection of Dedekind domains. Let $\mcl{J}_i\sbnq \mcl{O}_i$ be an infinite collection of non-zero ideals in each Dedekind domain $\mcl{O}_i,i\in \N$. Let $\mcl{S}=\us{i\in \N}{\prod}\mcl{O}_i,\mcl{J}=\us{i\in \N}{\prod}\mcl{J}_i$. Then $\mcl{J}$ is an ideal in the ring $\mcl{S}$ and $\frac{\mcl{S}}{\mcl{J}}\cong\us{i\in \N}{\prod} \frac{\mcl{O}_i}{\mcl{J}_i}$. Now the ideal $\mcl{J}$ is contained in infinitely many maximal ideals of $\mcl{S}$ even though $\mcl{J}_i$ is contained in only a finitely many maximal ideals of $\mcl{O}_i$ for each $i\in \N$. Since $\mcl{J}_i\sbnq \mcl{O}_i$ satisfies the USC for each $i\in \N$, we have $\mcl{J}\sbnq \mcl{O}$ satisfies the USC.	
	\end{enumerate}	
	
\end{example}
Now we mention an example of a ring $S$ and ideal $\langle 0\rangle \neq \mcl{J}\sbnq \R$ which does not satisfy the USC. Using Lemma $4.8$ in C.~P.~Anil Kumar~\cite{CPAKII}, we can conclude that the ideal $\langle x\rangle \sbnq \Z[x]$ does not satisfy the USC in $\Z[x]$. But we give a more difficult example below.

\begin{example}
	\label{Example:NotUSC}
	Let $\mcl{S}=\mbb{R}[x,y,z]$ be a polynomial ring in the three variables $x,y,z$. Let $\mcl{J}=\langle x^2+y^2+z^2-1\rangle$. Then clearly the ideal $\mcl{J}$ is not contained in only a finitely many maximal ideals of $\mcl{S}$. We show that $\mcl{J}$ does not satisfy the USC as well. If the ideal $\mcl{J}$ satisfies the USC then the zero ideal $\I=\langle 0\rangle \sbnq \R=\frac{\mcl{S}}{\I}$ satisfies the USC in $\R$ using Lemma $4.8$ in C.~P.~Anil Kumar~\cite{CPAKII}. Let $\mcl{I}_0=\langle 0\rangle,\mcl{I}_1=\R,\mcl{I}_2=\R, m^i_j=1, 0\leq i,j\leq 2$. The ideals $\I_0,\I_1,\I_2$ are mutually co-maximal and their product ideal is the zero ideal $\I$. Consider the map $SL_3(\R) \lra \mbb{PF}_{\mcl{I}_0}^{2,(1,1,1)} \times \mbb{PF}_{\mcl{I}_1}^{2,(1,1,1)} \times \mbb{PF}_{\mcl{I}_2}^{2,(1,1,1)}$. This map is surjective using Theorem $\Gom$ in~\cite{CPAKII}. But we have proved that this map is not surjective in Example $2.11$ in~\cite{CPAKII} which is a contradiction.
\end{example}
\begin{defn}
Let $k\in \N$. For $\R$, a commutative ring with unity, let $G_k(\R)$ be either the group $SL_{k+1}(\R)$ or the group $SP_{2k}(\R)$. An ideal $\mcl{J}\sbnq \R$ has the strong approximation property (SAP) with respect to $G_k(\R)$ if the reduction map $G_k(\R) \lra G_k(\frac{\R}{\mcl{J}})$ is surjective.
\end{defn}
\begin{remark}
See Section~\ref{sec:SAP} for more about SAP.
\end{remark}
\begin{defn}
	\label{defn:CIS}
Let $k\in \N$. For $\R$, a commutative ring with unity, let $G_k(\R)$ be either the group $SL_{k+1}(\R)$ or the group $SP_{2k}(\R)$. If $\R$ is the zero ring then define $G_k(\R)=\{1\}$. We say $\Gamma \subseteq G_k(\R)$  is a principal congruence ideal subgroup of $G_k(\R)$ if it is the kernel of the homomorphism $G_k(\R)\lra G_k(\frac{\R}{\mcl{J}})$ for some ideal $\mcl{J}\subseteq \R$. We denote here the subgroup $\Gamma$ by $\Gamma_k(\mcl{J})$. We say a subgroup $\Gamma \subseteq G_k(\R)$ is a congruence ideal subgroup of $G_k(\R)$ if it contains a principal congruence ideal subgroup $\Gamma_k(\mcl{J})$ for some ideal $\mcl{J}\subseteq\R$.
\end{defn}
\begin{remark}
In Definition~\ref{defn:CIS}, the ideal $\mcl{J}\sbnq \R$ need not have the SAP.  Depending on the context, the subgroup $\Gamma_k(\mcl{J})$ stands for the kernel of either the map $SL_{k+1}(\R)\lra SL_{k+1}(\frac{\R}{\mcl{J}})$ or the map $SP_{2k}(\R)\lra SP_{2k}(\frac{\R}{\mcl{J}})$.
\end{remark}

Now we state the two main results of the article.
\begin{thmOmega}
	\namedlabel{theorem:FullGenSurjOne}{$\Gom$}
Let $\R$ be a commutative ring with unity. 
Let $k\in \mbb{N}$ and $\mcl{I}_i,0\leq i\leq k$ be mutually co-maximal ideals in $\R$ such that either the ideal $\mcl{I}=\us{i=0}{\os{k}{\prod}}\mcl{I}_i$ satisfies the USC or $\mcl{I}=\R$. Let $m_j^i\in \mbb{N}, 0\leq i,j\leq k$ and $\Gamma \subseteq SL_{k+1}(\R)$ be a congruence ideal subgroup. Suppose  $\Gamma\Gamma_k(\I)=SL_{k+1}(\R)$. 
Then the map 
\equ{\Gamma \lra \us{i=0}{\os{k}{\prod}}\mbb{PF}^{k,(m^i_0,m^i_1,\ldots,m^i_k)}_{\I_i}} given by
\equa{&A_{(k+1)\times (k+1)}=[a_{i,j}]_{0\leq i,j\leq k} \lra\\
	&\big([a_{0,0}:a_{0,1}:\ldots: a_{0,k}],[a_{1,0}:a_{1,1}:\ldots: a_{1,k}],\ldots,[a_{k,0}:a_{k,1}:\ldots: a_{k,k}]\big)}
is surjective. In particular if $\Gamma\supseteq \Gamma_k(\mcl{J})$ a principal congruence ideal subgroup for some ideal $\mcl{J}$ such that $\mcl{J}+\mcl{I}=\R$ and either $\I\mcl{J}$ satisfies the USC or $\I\mcl{J}=\R$ then the map is surjective.
\end{thmOmega}
\begin{thmSigma}
	\namedlabel{theorem:FullGenSurjOneSP}{$\Gs$}
	Let $\R$ be a commutative ring with unity. 
	Let $k\in \mbb{N}$ and $\mcl{I}_i,1\leq i\leq 2k$ be mutually co-maximal ideals in $\R$ such that either the ideal $\mcl{I}=\us{i=1}{\os{2k}{\prod}}\mcl{I}_i$ satisfies the USC or $\mcl{I}=\R$.
	Let $m_j^i\in \mbb{N}, 1\leq i,j\leq 2k$ and $\Gamma \subseteq SP_{2k}(\R)$ be a congruence ideal subgroup. Suppose  $\Gamma\Gamma_k(\I)=SL_{2k}(\R)$.  
	Then the map 
	\equ{\Gamma \lra \us{i=1}{\os{2k}{\prod}}\mbb{PF}^{2k-1,(m^i_1,m^i_2,\ldots,m^i_{2k})}_{\I_i}} given by
	\equa{&A_{2k\times 2k}=[a_{i,j}]_{1\leq i,j\leq 2k} \lra\\
		&\big([a_{1,1}:a_{1,2}:\ldots: a_{1,2k}],[a_{2,1}:a_{2,2}:\ldots: a_{2,2k}],\ldots,[a_{2k,1}:a_{2k,2}:\ldots: a_{2k,2k}]\big)}
	is surjective. In particular if $\Gamma\supseteq \Gamma_k(\mcl{J})$ a principal congruence ideal subgroup for some ideal $\mcl{J}$ such that $\mcl{J}+\mcl{I}=\R$ and either $\I\mcl{J}$ satisfies the USC or $\I\mcl{J}=\R$ then the map is surjective.
\end{thmSigma}

\section{\bf{The Strong Approximation Property (SAP) and Generalized Euclidean Rings (GE-Rings)}}
\label{sec:SAP}
In this section we discuss some motivating examples about the SAP and prove a theorem with a converse relating the SAP and GE-Rings.

\begin{example}
	Let $\mcl{O}$ be a Dedekind domain and let $k\in \N$. Any ideal $\mcl{J}\sbnq \R$ has the SAP for both $SL_{k+1}(\mcl{O})$ and $SP_{2k}(\mcl{O})$. More generally, in a commutative ring $\R$ with unity, if an ideal $\mcl{J}\sbnq \R$ satisfies the USC, then it has the SAP with respect to $SL_{k+1}(\R)$ using Theorem $1.7$ on Page 338 in C.~P.~Anil Kumar~\cite{MR3887364} and it has the SAP with respect to $SP_{2k}(\R)$ using Theorem $\Gl$ in C.~P.~Anil Kumar~\cite{CPAKII}.   
\end{example}

We give another example of a ring $\R$ and an ideal $\I$ where the ideal does  satisfy the SAP with respect to $SL_{k+1}(\R),k\in \N$ but does not satisfy the USC in $\R$. So we cannot apply Theorem $1.7$ on Page $338$ in C.~P.~Anil Kumar~\cite{MR3887364}. 
\begin{example}
	\label{Example:NotUSC2} 
	Let $\R=\mbb{Z}[x]$ and $\I=\langle 5\rangle$. Now the set $\{x,3x^2-1\}$ is unital in $\Z[x]$. However $x+t(3x^2-1)$ is not a unit modulo $\I$ for any $t\in\Z[x]$. The units in $\mbb{F}_5[x]$ are non-zero constant polynomials which are of degree zero where as $x+t(3x^2-1)$ modulo $\langle 5\rangle$ is of degree at least $1$ which is a contradiction. So $\I$ does not satisfy the USC in $\Z[x]$.
	Since $\mbb{F}_5[x]$ is an Euclidean domain we have that $SL_{k+1}(\mbb{F}_5[x])$ is generated by elementary matrices of the form $E_{ij}(t)=I+te_{ij},1\leq i\neq j\leq k+1$ where $t\in \mbb{F}_5[x]$. This is because, to prove for any $k\in \N$, it is enough to express  the $2$-by-$2$ transposition determinant one matrix \equ{T=\mattwo 0{-1}10=\mattwo 1{-1}01\mattwo 1011\mattwo 1{-1}01} and the $2$-by-$2$ diagonal determinant matrix \equ{D(s)=\mattwo s00{s^{-1}}=\mattwo 10{s^{-1}}1\mattwo 1{1-s}01\mattwo 10{-1}1\mattwo 1{1-s^{-1}}01}
	and we can use Euclidean algorithm to prove this generation result. Now
	elementary matrices $E_{ij}(t),i\neq j$ are in the image of the map $SL_{k+1}(\Z[x]) \lra SL_{k+1}(\mbb{F}_5[x])$. Hence the map is surjective. Therefore the ideal $\I$ satisfies the SAP with respect to $SL_2(\R)$.
\end{example}
Thus we have the following theorem which we state without proof. The converse of this theorem is also given later after Remark~\ref{remark:GE}.
\begin{theorem}
	\label{theorem:SAP}
	Let $\R$ be a commutative ring with unity. Let $k\in \N$ and $\I\sbnq \R$ be a ideal.  Suppose $SL_k(\frac{\R}{\I})$ is generated by the elementary matrices of the form $E_{ij}(t)=I+te_{ij},t\in \frac{\R}{\I},i\neq j$. Then the reduction map $SL_k(\R)\lra SL_k(\frac{\R}{\I})$ is surjective, that is, the ideal $\I$ satisfies the SAP with respect to $SL_k(\R)$. 
\end{theorem}

\begin{remark}
	\label{remark:GE}
	For a commutative ring $\R$ with unity and $k\in \N$ we have that $GL_k(\R)$ is generated by the elementary matrices of the form $E_{ij}(t)=I+te_{ij},t\in \R,i\neq j$ and the diagonal matrices $D(s_1,s_2,\ldots,s_k)=\Diag(s_1,s_2,\ldots,s_k)$ where each $s_i,1\leq i\leq k$ is invertible in $\R$ if and only if $SL_k(\R)$ is generated by only the elementary matrices of the form $E_{ij}(t)=I+te_{ij},t\in \R,i\neq j$. Ring $\R$ which satisfies this generation condition for $GL_k(\R)$ or $SL_k(\R)$ is a $GE_k$-ring and if it satisfied for all $k\in \N$ then it is called a GE-ring, (GE stands for Generalized Euclidean). Refer to P.~M.~Cohn~\cite{MR0207856}. Examples of such commutative GE-rings are
	\begin{itemize}
		\item Euclidean domains (commutative), 
		\item the finite direct product of commutative GE-rings, 
		\item also commutative local rings as mentioned Section $4$ in~\cite{MR0207856}. Also refer to W.~Klingenberg~\cite{MR0143817}.
	\end{itemize}
	Examples of commutative rings which are not GE-rings are
	\begin{itemize}
		\item The ring of integers in an imaginary quadratic number field when the ring of integers is not a Euclidean domain.
		\item the polynomial ring $\Z[x_1,x_2,\ldots,x_n]$ in $n$-variables with integer coefficients where $n\in \N$,
		\item the polynomial ring $\mbb{K}[x_1,x_2,\ldots,x_n]$ in $n$-variables with coefficients in the field $\mbb{K}$ where $n\in \N$ and $n\geq 2$.   
	\end{itemize} 	
\end{remark}
Now we give a short proof of the converse of Theorem~\ref{theorem:SAP} under an additional assumption as follows.
\begin{theorem}
	\label{theorem:SAPConverse}
Let $\R$ be a commutative ring with unity, Let $k\in \N$ and $\I\sbnq \R$ be an ideal. Suppose the ideal $\I$ satisfies the SAP with respect to $SL_k(\R)$ and $SL_k(\R)$ is generated by only the elementary matrices of the form $E_{ij}(t)=I+te_{ij},t\in \R,i\neq j$. Then $SL_k(\frac{\R}{\I})$ is generated by only the elementary matrices of the form $E_{ij}(t)=I+te_{ij},t\in \frac{\R}{\I},i\neq j$.
\end{theorem} 
\begin{proof}
Let $A\in SL_k(\frac{\R}{\I})$. Then by the SAP we have that there exists $B\in SL_k(\R)$ whose image is the matrix $A$. Now express $B$ as the product of elementary matrices. The image of this product expresses $A$ as the product of elementary matrices. This proves the theorem.
\end{proof}

\begin{remark}
Let $\R$ be a commutative ring with unity. Let $k\in \N$ and $\I\sbnq \R$ be an ideal.	
If $\frac{\R}{\I}$ is a $GE_k$-ring then $\I$ satisfies the SAP with respect to $SL_k(\R)$. Conversely if $\I$ satisfies the SAP with respect to $SL_k(\R)$ and $\R$ is a $GE_k$-ring then $\frac{\R}{\I}$ is a $GE_k$-ring.
\end{remark}
Now we mention an example of a ring $\R$ and ideal $\langle 0\rangle\neq  \I\sbnq \R$ such that the ideal $\I$ does not have the SAP with respect to $SL_3(\R)$. Hence in particular the ideal $\I$ does not satisfy the USC as well. Here the ring $\frac{\R}{\I}$ must not be a $GE_3$-ring using Theorem~\ref{theorem:SAP}. 
\begin{example}
\label{Example:NotSAP}
Let $\R=\mbb{R}[x,y]$ and $\I=\langle x^2+y^2-1\rangle$. The group $SL_3(\R)$ is generated by only the elementary matrices of the form $E_{ij}(t)=I+te_{ij},t\in \R,i\neq j$.	For a proof of this fact see A.~A.~Suslin~\cite{MR0472792}. So $\R$ is a $GE_3$-ring. Now the ring $SL_3(\frac{\R}{\I})$ is not generated by only the elementary matrices of the form $E_{ij}(t)=I+te_{ij},t\in \R,i\neq j$, that is, $\frac{\R}{\I}$ is not a $GE_3$-ring.
For example the matrix \equ{\matthree {\ol{x}}{-\ol{y}}0{\ol{y}}{\ol{x}}0001\in SL_3(\frac{\R}{\I})}
is not a product of elementary matrices. For a proof of this fact see proof of Proposition $8.12$ in Chapter I of T.~Y.~Lam~\cite{MR2235330} on Pages $57-58$. So the map $SL_3(\R)\lra SL_3(\frac{\R}{\I})$ is not surjective. Hence the ideal $\I$ does not have the SAP with respect to $SL_3(\R)$. So the ideal $\I=\langle x^2+y^2-1\rangle \sbnq \mbb{R}[x,y]$ does not satisfy the USC.
\end{example}

\section{\bf{Proof of the Main Theorems}}
We begin with a lemma.
\begin{lemma}
	\label{lemma:Complement}
Let $\R$ be a commutative ring with unity and let $\I,\mcl{J}\sbnq \R$ be two co-maximal ideals such that the product ideal $\I\mcl{J}$ satisfies the USC. Let $k\in \N$ and $G_k(\R)$ be either the group $SL_{k+1}(\R)$ or the group $SP_{2k}(\R)$. Then the two principal congruence normal subgroups $\Gamma_k(\I),\Gamma_k(\mcl{J}) \trianglelefteq  G_k(\R)$ satisfy the following.
\begin{enumerate}
\item $\Gamma_k(\I)\cap \Gamma_k(\mcl{J})=\Gamma_k(\I\mcl{J})$,
\item $\Gamma_k(\mcl{I})\Gamma_k(\mcl{J})=G_k(\R)$.
\item $G_k(\frac{\R}{\I})\cong \Gamma_k(\mcl{J}),G_k(\frac{\R}{\mcl{J}})\cong \Gamma_k(\I),G_k(\frac{\R}{\I\mcl{J}})\cong \Gamma_k(\I)\oplus \Gamma_k(\mcl{J})$
\end{enumerate}
\end{lemma}
\begin{proof}
The maps $G_k(\R)\lra G_k(\frac{\R}{\I}),G_k(\R)\lra G_k(\frac{\R}{\mcl{J}}),G_k(R)\lra G_k(\frac{\R}{\I\mcl{J}})$ because the three ideals $\I,\mcl{J},\I\mcl{J}$ satisfy SAP as they satisfy the USC using Lemma $4.7$ in C.~P.~Anil Kumar~\cite{CPAKII}. Moreover, since $\I,\mcl{J}$ are co-maximal we have that the map $G_k(\frac{\R}{\I\mcl{J}})\lra G_k(\frac{\R}{\mcl{I}})\oplus G_k(\frac{\R}{\mcl{J}})$ is an isomorphism using Chinese remainder theorem for the pair of co-maximal ideals $\I,\mcl{J}$. Hence we obtain $\Gamma_k(\I)\cap \Gamma_k(\mcl{J})=\Gamma_k(\I\mcl{J})$. This proves $(1)$. Now we have a composite surjection $G_k(\R)\lra G_k(\frac{\R}{\I\mcl{J}}) \lra G_k(\frac{\R}{\mcl{I}})\oplus G_k(\frac{\R}{\mcl{J}})=\frac{G_k(\R)}{\Gamma_k(\I)}\oplus \frac{G_k(\R)}{\Gamma_k(\mcl{J})}$. Let $g\in G_k(\R)$. Then there exists $h\in G_k(\R)$ such that $h\Gamma_k(\I)=g\Gamma_k(\I)$ and $h\Gamma_k(\mcl{J})=\Gamma_k(\mcl{J})$. So $h\in \Gamma_k(\mcl{J})$ and $gh^{-1}=k\in \Gamma_k(\I)$. Hence $g=kh\in \Gamma_k(\mcl{I})\Gamma_k(\mcl{J})$. This proves $(2)$. Now $(3)$ also follows. This proves the lemma.	
\end{proof}
Now we prove two theorems which are useful in proving the main results.
\begin{theorem}
	\label{theorem:GenSurjMainGenIdeals}
	Let $\R$ be a commutative ring with unity. 
	Let $k\in \mbb{N}$ and $\mcl{I}_0,\mcl{I}_1,\ldots,\mcl{I}_k$ be $(k+1)$ mutually co-maximal ideals in $\R$. Let $\mcl{I}=\us{i=0}{\os{k}{\prod}}\mcl{I}_i$ be such that $\I$ satisfies the USC or $\I=\R$. Let $\Gamma \subseteq SL_{k+1}(\R)$ be a  subgroup such that $\Gamma\Gamma_k(\mcl{\I})=SL_{k+1}(\R)$.
	Let $A_{(k+1)\times (k+1)}=[a_{i,j}]_{0\leq i,j \leq k}\in M_{(k+1)\times (k+1)}(\R)$ be such that for every $0\leq i\leq k$ the $i^{th}\operatorname{-}$row is unital, that is, $\us{j=0}{\os{k}{\sum}} \langle a_{i,j}\rangle =\R$
	for $0\leq i\leq k$. Then there exists $B=[b_{i,j}]_{0\leq i,j \leq k}\in \Gamma$ such that we have 
	$a_{i,j}\equiv b_{i,j}\mod \mcl{I}_i,0\leq i,j\leq k$. In particular $\Gamma$ can be a congruence ideal subgroup such that $\Gamma\supseteq \Gamma_k(\mcl{J})$ a principal congruence ideal subgroup for some ideal $\mcl{J}$ such that $\mcl{J}+\mcl{I}=\R$ and either $\I\mcl{J}$ satisfies the USC or $\I\mcl{J}=\R$.
\end{theorem}
\begin{proof}
Now we have either $\I=\R$ or $\I$ satisfies the USC using Lemma $4.7$ in C.~P.~Anil Kumar~\cite{CPAKII}. Now we can clearly assume $\I\neq \R$. So using Theorem $7.1$ in~\cite{CPAKII} there exists $X=[x_{i,j}]_{0\leq i,j \leq k}\in SL_{k+1}(\R)$ such that $a_{i,j}\equiv x_{i,j}\mod \mcl{I}_i,0\leq i,j\leq k$. Since $\Gamma\Gamma_k(\I)=\Gamma_k(\I)\Gamma=SL_{k+1}(\R)$ we have $X=YB$ for some $Y\in\Gamma_k(\I),B\in \Gamma$.  Therefore if $B=[b_{i,j}]_{0\leq i,j \leq k}\in \Gamma$ then we have 
$a_{i,j}\equiv b_{i,j}\mod \mcl{I}_i,0\leq i,j\leq k$.
 
In case $\Gamma \supseteq \Gamma_k(\mcl{J})$, then using Lemma~\ref{lemma:Complement} we have $\Gamma_k(\I)\Gamma_k(\mcl{J})=SL_{k+1}(\R)=\Gamma\Gamma_k(\I)$.  Hence the theorem follows for $\Gamma$.
\end{proof}
The corresponding theorem for $SP_{2k}(\R),k\in \N$ is as follows.
\begin{theorem}
	\label{theorem:GenSurjMain}
	Let $\R$ be a commutative ring with unity. 
	Let $k\in \mbb{N}$ and $\mcl{I}_1,\mcl{I}_1,\ldots,\mcl{I}_{2k}$ be $2k$ mutually co-maximal ideals in $\R$. Let $\mcl{I}=\us{i=1}{\os{2k}{\prod}}\mcl{I}_i$ be such that either $\I$ satisfies the USC or $\I=\R$. Let $\Gamma \subseteq SP_{2k}(\R)$ be a  subgroup such that $\Gamma\Gamma_k(\mcl{\I})=SP_{2k}(\R)$.
	Let $M_{(2k)\times (2k)}=[m_{i,j}]_{1\leq i,j \leq 2k}\in M_{(2k)\times (2k)}(\R)$ such that for every $1\leq i\leq 2k$ the $i^{th}\operatorname{-}$row is unital, that is, $\us{j=1}{\os{2k}{\sum}} \langle m_{i,j}\rangle =\R$
	for $1\leq i\leq 2k$. Then there exists $N=[n_{i,j}]_{1\leq i,j \leq 2k}\in \Gamma$ such that we have 
	$m_{i,j}\equiv n_{i,j}\mod \mcl{I}_i,1\leq i,j\leq 2k$. In particular $\Gamma$ can be a congruence ideal subgroup such that $\Gamma\supseteq \Gamma_k(\mcl{J})$ a principal congruence ideal subgroup for some ideal $\mcl{J}$ such that $\mcl{J}+\mcl{I}=\R$ and either $\I\mcl{J}$ satisfies the USC or $\I\mcl{J}=\R$.
\end{theorem}
\begin{proof}
The proof is similar to the proof of Theorem~\ref{theorem:GenSurjMainGenIdeals} except that here we use Theorem $8.1$ in~\cite{CPAKII}.	
\end{proof}
Now we prove the main Theorem~\ref{theorem:FullGenSurjOne}.
\begin{proof}
Let \equa{\big([a_{0,0}:a_{0,1}:\ldots: a_{0,k}],[a_{1,0}:a_{1,1}:\ldots: a_{1,k}],\ldots,&[a_{k,0}:a_{k,1}:\ldots: a_{k,k}]\big)\\&\in \us{i=0}{\os{k}{\prod}}\mbb{PF}^{k,(m^i_0,m^i_1,\ldots,m^i_k)}_{\I_i}.}
Consider $A_{(k+1)\times (k+1)}=[a_{i,j}]_{0\leq i,j\leq k}\in M_{(k+1)\times (k+1)}(\R)$ for which Theorem~\ref{theorem:GenSurjMainGenIdeals} can be applied. Therefore we get 
$B=[b_{i,j}]_{0\leq i,j\leq k} \in SL_{k+1}(\R)$ such that $b_{i,j}\equiv a_{i,j} \mod \mcl{I}_i,0\leq i,j\leq k$. Hence we get 
\equ{[b_{i,0}:b_{i,1}:\ldots: b_{i,k}]=[a_{i,0}:a_{i,1}:\ldots: a_{i,k}]\in \mbb{PF}^{k,(m^i_0,m^i_1,\ldots,m^i_k)}_{\I_i},0\leq i\leq k.}
In case $\Gamma \supseteq \Gamma_k(\mcl{J})$, then using Lemma~\ref{lemma:Complement} we have $\Gamma_k(\I)\Gamma_k(\mcl{J})=SL_{k+1}(\R)=\Gamma_k(\I)\Gamma$.
This proves the main Theorem~\ref{theorem:FullGenSurjOne}.	
\end{proof}
Now we prove the main Theorem~\ref{theorem:FullGenSurjOneSP}.
\begin{proof}
	Let \equa{\big([a_{1,1}:a_{1,2}:\ldots: a_{1,2k}],[a_{2,1}:a_{2,2}:\ldots: a_{2,2k}],\ldots,&[a_{2k,1}:a_{2k,2}:\ldots: a_{2k,2k}]\big)\\&\in \us{i=1}{\os{2k}{\prod}}\mbb{PF}^{2k-1,(m^i_1,m^i_2,\ldots,m^i_{2k})}_{\I_i}.}
	Consider $A_{2k\times 2k}=[a_{i,j}]_{1\leq i,j\leq 2k}\in M_{2k\times 2k}(\R)$ for which Theorem~\ref{theorem:GenSurjMain} can be applied. Therefore we get 
	$B=[b_{i,j}]_{1\leq i,j\leq 2k} \in \Gamma$ such that $b_{i,j}\equiv a_{i,j} \mod \mcl{I}_i,1\leq i,j\leq 2k$. Hence we get for $1\leq i\leq 2k$,
	\equ{[b_{i,1}:b_{i,2}:\ldots: b_{i,2k}]=[a_{i,1}:a_{i,2}:\ldots: a_{i,2k}]\in \mbb{PF}^{2k-1,(m^i_1,m^i_2,\ldots,m^i_{2k})}_{\I_i}.}
In case $\Gamma \supseteq \Gamma_k(\mcl{J})$, then using Lemma~\ref{lemma:Complement} we have $\Gamma_k(\I)\Gamma_k(\mcl{J})=SP_{2k}(\R)=\Gamma_k(\I)\Gamma$.
	This proves the main Theorem~\ref{theorem:FullGenSurjOneSP}.
\end{proof}
\bibliographystyle{abbrv}
\def\cprime{$'$}

\end{document}